\documentclass[11pt,dvips,twoside,letterpaper]{article}

\usepackage{pslatex}
\usepackage{fancyhdr}
\usepackage{graphicx}
\usepackage{geometry}

\def\figurename{Figure} 
\makeatletter
\renewcommand{\fnum@figure}[1]{\figurename~\thefigure.}
\makeatother

\def\tablename{Table} 
\makeatletter
\renewcommand{\fnum@table}[1]{\tablename~\thetable.}
\makeatother

\usepackage{amsmath}
\usepackage{amssymb}
\usepackage{amsfonts}
\usepackage{amsthm,amscd}

\newtheorem{theorem}{Theorem}[section]

\newtheorem{proposition}[theorem]{Proposition}
\theoremstyle{definition}

\theoremstyle{remark}
\newtheorem{remark}[theorem]{Remark}

\numberwithin{equation}{section}

\begin{document}
\title{\bfseries\scshape{Factorization of linear and nonlinear differential operators: necessary and sufficient conditions}}
\author{\bfseries\scshape Mahouton Norbert Hounkonnou\thanks{E-mail address: norbert.hounkonnou@cipma.uac.bj}\\
International Chair in Mathematical Physics
and Applications\\ (ICMPA-UNESCO Chair)\\University of Abomey-Calavi\\072 B.P. 50  Cotonou, Republic of Benin\\ \\
\bfseries\scshape Pascal Dkengne Sielenou\thanks{E-mail address: sielenou$\_$alain@yahoo.fr}\\
International Chair in Mathematical Physics
and Applications\\ (ICMPA-UNESCO Chair)\\University of Abomey-Calavi\\072 B.P. 50  Cotonou, Republic of Benin\\ \\
 \\ \\}

\date{}
\maketitle \thispagestyle{empty} \setcounter{page}{1}

\thispagestyle{fancy} \fancyhead{}
\renewcommand{\headrulewidth}{0pt}

\begin{abstract} \noindent
An algebraic approach for  factorizing  nonlinear partial differential equations (PDEs) and systems of PDEs is provided.
 In the particular case of second order linear and nonlinear
  PDEs and systems of PDEs, necessary and sufficient conditions of factorization are given.
\end{abstract}


\vspace{.08in} \noindent \textbf{Keywords}: Linear differential equations; nonlinear differential equations; factorization.

\section{Introduction}

The search for exact solutions of differential equations is very challenging in mathematics, but their usefulness   in the proper understanding of
qualitative features of  phenomena and processes in various areas of natural science merits to get down to such an investigation.
Indeed, exact solutions can be used
to verify the consistency and estimate errors of various numerical, asymptotic and approximate analytical methods.
Unfortunately, there does not always exist a method adapted for the resolution of any  type of differential equations.
Very often, one tries to  reduce the equation in order to make easier its resolution.
But this reduction requires the knowledge of suitable transformations or changes of  variables. The latters
  usually give rise to another problem the issue of which
 is not always favourable.

A simple approach for the reduction of a differential equation consists in seeking a factorization, if there exists,
of the differential operator associated with it.
Note that for the particular case of second order linear  ordinary differential equations   of   Schr\"{o}dinger or Sturm-Liouville type,
the factorization of the associated differential operators also allows to obtain partially or completely their spectrum,
under certain assumptions of integrability \cite{r1,r2,r3,r4}.
In recent years, there has been much interest devoted to  the problem of factorization of  differential equations, especially
based on linear  ordinary
 \cite{r5,r6,r07} and nonlinear differential operators \cite{r7,r07,r8}.  Although effective, the used methods  are rather restrictive in their applications.

Recently, a  purely algebraic method of factorization of the second order
linear  ordinary differential equations  has been presented by the
authors in \cite{r9,r10,r11,r12}.
The same procedure of factorization has been exploited in \cite{r13} and  extended to  second order
nonlinear  ordinary differential equations (NLODES) and systems of NLODEs.
This work generalizes previous  works   by applying the above mentioned algebraic method of factorization
to linear and nonlinear systems of partial differential equations (PDEs). Necessary and sufficient conditions of factorization
are derived in the case of
 second order equations.

First of all, some useful notations are required.
Consider $X,$ an $n$-dimensional independent variable space, and $U,$
 an $m$-dimensional dependent variable space.
 Let $x=\left(x^1,\cdots,x^n\right)\in X$ and $u=\left(u^1,\cdots,u^m\right)\in U.$
  We define  the space $U^{(s)},$ $s\in \mathbb{N}$ as:
  \begin{equation}
   U^{(s)}:=\left\{u^{(s)}\,:\,\,u^{(s)}=\bigotimes_{j=1}^{m}\left(\bigotimes_{k=0}^{s}u^j_{(k)}  \right)\right\},
  \end{equation}
 where $u^j_{(k)}$ is the
 \begin{equation}\label{sa1}
 p_k= n^k
 \end{equation}
  of all  $k$-th order partial derivatives of $u^j.$
The   $u^j_{(k)}$ vector components are recursively obtained  as follows:
 \begin{itemize}
 \item[i)] $u^j_{(0)}=u^j$ and $u^j_{(1)}=\left(u^j_{x^1},u^j_{x^2},\cdots,u^j_{x^n}\right);$
 \item[ii)] Assume that  $u^j_{(k)}$ is known. Then,
 \begin{itemize}
\item Form the tuples $\widetilde{u}^j_{(k+1)}(l)$ as follows:
 $$\widetilde{u}^j_{(k+1)}(l)=\left(\frac{\partial}{\partial x^1}u^j_{(k)}[l], \frac{\partial}{\partial x^2}u^j_{(k)}[l],\cdots,
\frac{\partial}{\partial x^n}u^j_{(k)}[l]   \right),\quad l=1,2,\cdots,p_k,$$
 where $u^j_{(k)}[l]$ is the $l$-th component of the vector $u^j_{(k)};$
\item Finally, form the vector
$$u^j_{(k+1)}=\left( \widetilde{u}^j_{(k+1)}(1),\widetilde{u}^j_{(k+1)}(2),\cdots,\widetilde{u}^j_{(k+1)}(p_k)   \right).  $$
\end{itemize}
 \end{itemize}
  An element   $u^{(s)},$ in the space $U^{(s)},$ is the
  \begin{equation}\label{sa2}
  q_s= m(1+p_1+p_2+\cdots+p_s)\emph{\emph{-tuple}}
  \end{equation}
  defined by
\begin{equation}\label{sa3}
u^{(s)}=\left(u^1_{(0)},u^1_{(1)},\cdots,u^1_{(s)}, u^2_{(0)},u^2_{(1)},\cdots,u^2_{(s)},\cdots,u^m_{(0)},u^m_{(1)},\cdots,u^m_{(s)}  \right).
\end{equation}
The coordinates in the space $X\times U^{(s)}$   are
denoted by $\left(x,u^{(s)}\right).$ \\
In the sequel, the $q_s$-uple $u^{(s)}$ will be referred to (\ref{sa3}), whereas the integers $p_k$ and $q_s$ are  defined by (\ref{sa1}) and (\ref{sa2}), respectively.
Define differential operators $\emph{\emph{D}}_{k,h}$   whose action on a regular  function $u$ is
\begin{equation}\label{eqa1}
\emph{\emph{D}}_{k,h} \,u=u_{(k)}[h]
\end{equation}
These  operators
$\emph{\emph{D}}_{k,h}$ satisfy the following properties:
\begin{itemize}
\item[(i)] $\emph{\emph{D}}_{0,1}\,u=u$ $\qquad$  (identity),
\item[(ii)] $\emph{\emph{D}}_{1,h} \,\emph{\emph{D}}_{k',h'} \,u=\emph{\emph{D}}_{k'+1,n(h'-1)+h}\,u$  $\qquad$ (composition rule),
\item[(iii)] $\emph{\emph{D}}_{k,h}\,u=\emph{\emph{D}}_{1,h}\,\emph{\emph{D}}_{k-1,1}\,u,$ $\quad k\geq 1$ $\qquad$ (decomposition rule).
\end{itemize}
\begin{remark}
Operators $\emph{\emph{D}}_{k,h}$ allow the simplification of the writing of certain differential operators. For example,
the operator
$$
\mathcal{T}=\sum_{l_1+l_2+\cdots+l_n=0}^s\frac{\partial^{l_1+l_2+\cdots+l_n}}{\left(\partial x^1  \right)^{l_1}\left(\partial x^2  \right)^{l_2}\cdots \left(\partial x^n  \right)^{l_n}}
$$
can  be shortly expressed as
$$
\mathcal{T}=\sum_{k=0}^{s}\sum_{h=1}^{p_k}\emph{D}_{k,h}.
$$
\end{remark}

\pagestyle{fancy} \fancyhead{} \fancyhead[EC]{M. N. Hounkonnou and P. A.  Dkengne } \fancyhead[EL,OR]{\thepage}
\fancyhead[OC]{Factorization of  differential operators}
\fancyfoot{}
\renewcommand\headrulewidth{0.5pt}

\section{Linear  differential operators}
In this section, we develop an algebraic method of factorization applicable to linear differential operators (LDOs) and to systems of LDOs.

\subsection{Factorizations of linear  differential equations}
The general setting of the factorization problem for LDOs is developed. Necessary and sufficient conditions are derived for the factorization of second order linear
ordinary and partial differential operators with two independent variables.
\subsubsection{General setting}

Let $s\geq 2$ be a positive integer and $\Lambda$ be an open subset of $\mathbb{R}^n.$
Let
\begin{equation}\label{eqa2}
\mathcal{P}(s)=\sum_{k=0}^{s}\sum_{h=1}^{p_k}g_{k,h}(x)\emph{\emph{D}}_{k,h}
\end{equation}
a linear differential operator of order $s,$ where $g_{k,h}\in \mathcal{C}(\Lambda,\mathbb{R}).$
The operator $\mathcal{P}(s)$ acts on a function $u\in \mathcal{C}^{s}(\Lambda,\mathbb{R})$ as follows
\begin{equation}\label{eqa3}
\mathcal{P}(s)\,u=\sum_{k=0}^{s}\sum_{h=1}^{p_k}g_{k,h}(x)\emph{\emph{D}}_{k,h}\,u.
\end{equation}
The method of factorization consists in seeking a decomposition of the differential operator (\ref{eqa2})
in the following form
\begin{equation}\label{eqa4}
\mathcal{P}(s)=\prod_{i=1}^{l}\mathcal{Q}_i(s_i)
\end{equation}
with $\sum_{i=1}^{l}s_i=s$ and
\begin{equation}\label{eqa5}
\mathcal{Q}_i(s_i)=\sum_{k=0}^{s_i}\sum_{h=1}^{p_k}b_{i,k,h}(x)\emph{\emph{D}}_{k,h},
\end{equation}
where $b_{1,k,h}\in \mathcal{C}(\Lambda,\mathbb{R})$  and  $b_{i,k,h}\in \mathcal{C}^{\sum_{j=1}^{i-1}s_j}(\Lambda,\mathbb{R}),$
$i=2,3, \cdots,l.$
\begin{proposition}\label{pn1}
Let $\mathcal{P}(s)$ be an operator which can be decomposed into the form
 (\ref{eqa4}). If the function $u_0$  satisfies
\begin{equation}\label{eqa6}
\mathcal{Q}_l(s_l)u_0=0,
\end{equation}
and $u_1,\,\ldots,\,u_{l-1}$ are solutions of the system
\begin{eqnarray}
 \label{eqa7}
  \prod_{k=l-j+1}^{l}\mathcal{Q}_{k}(s_k)u_j&=& v_{j},\;j=1,\,2,\ldots,\,l-1,
\end{eqnarray}
where $v_j,$ $j=1,\,2,\ldots,\,l-1,$ are solutions of
\begin{eqnarray}
\label{eqa8}
\prod_{i=1}^{l-j}\mathcal{Q}_i(s_i)v_j&=&0,
\end{eqnarray}
then $u_0,$ $u_1,\,\ldots,\,u_{l-1}$ are $l$ particular solutions of the equation
$\mathcal{P}(s)\,u=0.$
\end{proposition}
\begin{proof}
Let $u_0$ and $u_j,$ $j=1,\,2,\ldots,\,l-1$ be solutions of (\ref{eqa6}) and (\ref{eqa7}), respectively.
Then
$$ \mathcal{P}(s)u_0=\left(\prod_{i=1}^{l-1}\mathcal{Q}_i(s_i)\right)\mathcal{Q}_l(s_l)u_0=0,$$
and for  $j=1,\,2,\ldots,\,l-1,$
\begin{eqnarray}
 \mathcal{P}(s)u_j&=&\left(\prod_{i=1}^{l-j}\mathcal{Q}_i(s_i)\right)\left(\prod_{k=l-j+1}^{l}\mathcal{Q}_k(s_k)\right)u_j\nonumber\\
 \label{eqp5}
 &=&\prod_{i=1}^{l-j}\mathcal{Q}_i(s_i)v_j=0,\nonumber
\end{eqnarray}
where the use of (\ref{eqa7}) and (\ref{eqa8}) has been made.
\end{proof}
Expanding (\ref{eqa4}) leads to the relations between unknown
functions $b_{i,k,h}$ of the differential operators
$\mathcal{Q}_i(s_i)$ and the known functions $g_{k,h}$ of
the original differential operator $\mathcal{P}(s).$
Without loss of generality and as matter of clarity, this study will be concentrated to second order equations, the generalization being straightforward.

\subsubsection{Necessary and sufficient conditions for the factorization of second order linear ODEs}

Let $\Lambda$ and $\Lambda_0$ be two open subsets of $\mathbb{R}$ such that $\Lambda_0 \subset  \Lambda.$
Consider the second order linear ordinary differential operator
\begin{eqnarray}\label{eqa9}
\mathcal{P}(2)&=&\sum_{k=0}^{2}\sum_{h=1}^{p_k}g_{k,h}(x)\emph{\emph{D}}_{k,h}\nonumber\\
&=& g_{0,1}(x)\emph{\emph{D}}_{0,1}+g_{1,1}(x)\emph{\emph{D}}_{1,1}
+g_{2,1}(x)\emph{\emph{D}}_{2,1},
\end{eqnarray}
where $g_{k,h}\in \mathcal{C}(\Lambda,\mathbb{R})$ and  $x=x^1.$
Write $\mathcal{P}(2)$ in the form
\begin{eqnarray}\label{eqa10}
\mathcal{P}(2)&=&\mathcal{Q}_1(1)\cdot\mathcal{Q}_2(1)\nonumber\\
&=&\left[\sum_{k=0}^{1}\sum_{h=1}^{p_k}b_{1,k,h}(x)\emph{\emph{D}}_{k,h}\right]
\left[\sum_{k=0}^{1}\sum_{h=1}^{p_k}b_{2,k,h}(x)\emph{\emph{D}}_{k,h}\right]\nonumber\\
&=&\left[ b_{1,0,1}(x)\emph{\emph{D}}_{0,1}+b_{1,1,1}(x)\emph{\emph{D}}_{1,1}  \right]
\left[  b_{2,0,1}(x)\emph{\emph{D}}_{0,1}+b_{2,1,1}(x)\emph{\emph{D}}_{1,1} \right],
\end{eqnarray}
where $b_{1,k,h}\in \mathcal{C}(\Lambda,\mathbb{R})$  and  $b_{2,k,h}\in \mathcal{C}^{1}(\Lambda,\mathbb{R}).$
Let $u\in \mathcal{C}^2(\Lambda_0,\mathbb{R}).$ Then we have
\begin{eqnarray}\label{eqa11}
\mathcal{P}(2)\,u
&=& g_{0,1}(x)\,u+g_{1,1}(x)\,u_x
+g_{2,1}(x)\,u_{2x}
\end{eqnarray}
and after expansion
\begin{eqnarray}\label{eqa12}
\mathcal{P}(2)\,u&=&\left[ b_{1,0,1}(x)\emph{\emph{D}}_{0,1}+b_{1,1,1}(x)\emph{\emph{D}}_{1,1}  \right]
\left[  b_{2,0,1}(x)\emph{\emph{D}}_{0,1}+b_{2,1,1}(x)\emph{\emph{D}}_{1,1} \right]u\nonumber\\
&=&b_{1,1,1}b_{2,1,1}\,u_{2x}
+\left[ b_{1,0,1}b_{2,1,1}+b_{1,1,1}b_{2,0,1}+b_{1,1,1}\emph{\emph{D}}_{1,1}(b_{2,1,1})  \right]u_x\nonumber\\
&+&\left[ b_{1,0,1}b_{2,0,1}+b_{1,1,1}\emph{\emph{D}}_{1,1}(b_{2,0,1})  \right]u.
\end{eqnarray}
Identifying (\ref{eqa11}) with (\ref{eqa12}) yields
\begin{proposition}
A necessary and sufficient condition for the differential operator $\mathcal{P}(2) $ defined by (\ref{eqa9})  be decomposed into
the form (\ref{eqa10}) is:
\begin{eqnarray}
g_{2,1} &=& b_{1,1,1}b_{2,1,1},
\label{eqa13}\\
g_{1,1} &=& b_{1,0,1}b_{2,1,1}+b_{1,1,1}b_{2,0,1}+b_{1,1,1}\emph{\emph{D}}_{1,1}(b_{2,1,1}) ,\label{eqa14}\\
g_{0,1} &=& b_{1,0,1}b_{2,0,1}+b_{1,1,1}\emph{\emph{D}}_{1,1}(b_{2,0,1}).\label{eqa15}
\end{eqnarray}
\end{proposition}
Propose an approach to solve system (\ref{eqa13})-(\ref{eqa15}). Assume that $g_{2,1}$ does not  vanish on $\Lambda.$
Thus, it is always possible to find two nonzero functions on $\Lambda,$ namely $b_{1,1,1}$ and $b_{2,1,1},$ which satisfy (\ref{eqa13}).
Substituting $X=b_{1,0,1}$ and $Y=b_{2,0,1}$  in (\ref{eqa14}) gives
\begin{equation}\label{eqa16}
X=\frac{1}{b_{2,1,1}}\left[ g_{1,1}- b_{1,1,1}\emph{\emph{D}}_{1,1}(b_{2,1,1})- b_{1,1,1}Y  \right].
\end{equation}
The substitution of (\ref{eqa16}) into (\ref{eqa15}) implies that the decomposition (\ref{eqa10})  is strongly related
to the existence of a solution to the following Riccati equation in $Y$
 \begin{equation}\label{eqa17}
\emph{\emph{D}}_{1,1}(Y)-\frac{b_{1,1,1}}{g_{2,1}}Y^2+\frac{g_{1,1}-b_{1,1,1}\emph{\emph{D}}_{1,1}(b_{2,1,1})}{g_{2,1}}Y-\frac{g_{0,1}}{b_{1,1,1}}=0.
\end{equation}

\subsubsection{Necessary and sufficient conditions for the factorization of second order linear PDEs with two independent variables}

Let $\Lambda$ and $\Lambda_0$ be two open subsets of $\mathbb{R}^2$ such that $\Lambda_0 \subset  \Lambda.$
Consider the second order linear partial differential operator
\begin{eqnarray}\label{eqaa9}
\mathcal{P}(2)&=&\sum_{k=0}^{2}\sum_{h=1}^{p_k}g_{k,h}(x)\emph{\emph{D}}_{k,h}\nonumber\\
&=& g_{0,1}(x)\emph{\emph{D}}_{0,1}+g_{1,1}(x)\emph{\emph{D}}_{1,1}+g_{1,2}(x)\emph{\emph{D}}_{1,2}\nonumber\\
&+&g_{2,1}(x)\emph{\emph{D}}_{2,1}+g_{2,2}(x)\emph{\emph{D}}_{2,2}+g_{2,3}(x)\emph{\emph{D}}_{2,3}+g_{2,4}(x)\emph{\emph{D}}_{2,4},
\end{eqnarray}
where $g_{k,h}\in \mathcal{C}(\Lambda,\mathbb{R})$ and $x=\left(x^1,x^2\right).$
Write $\mathcal{P}(2)$ in the form
\begin{eqnarray}\label{eqaa10}
\mathcal{P}(2)&=&\mathcal{Q}_1(1)\cdot\mathcal{Q}_2(1)\nonumber\\
&=&\left[\sum_{k=0}^{1}\sum_{h=1}^{p_k}b_{1,k,h}(x)\emph{\emph{D}}_{k,h}\right]\left[\sum_{k=0}^{1}\sum_{h=1}^{p_k}b_{2,k,h}(x)\emph{\emph{D}}_{k,h}\right]\nonumber\\
&=&\left[ b_{1,0,1}(x)\emph{\emph{D}}_{0,1}+b_{1,1,1}(x)\emph{\emph{D}}_{1,1} +b_{1,1,2}(x)\emph{\emph{D}}_{1,2} \right]\nonumber\\
&\times&
\left[  b_{2,0,1}(x)\emph{\emph{D}}_{0,1}+b_{2,1,1}(x)\emph{\emph{D}}_{1,1} +b_{2,1,2}(x)\emph{\emph{D}}_{1,2}\right],
\end{eqnarray}
where $b_{1,k,h}\in \mathcal{C}(\Lambda,\mathbb{R})$  and  $b_{2,k,h}\in \mathcal{C}^{1}(\Lambda,\mathbb{R}).$
Let $u\in \mathcal{C}^2(\Lambda_0,\mathbb{R}).$ Then we have
\begin{eqnarray}\label{eqaa11}
\mathcal{P}(2)\,u
&=& g_{0,1}\,u+g_{1,1}\,u_{x^1}+g_{1,2}\,u_{x^2}
+g_{2,1}\,u_{2x^1}+\left(g_{2,2}+g_{2,3}\right)\,u_{x^1x^2}+g_{2,4}\,u_{2x^2}
\end{eqnarray}
and after expansion
\begin{eqnarray}\label{eqaa12}
\mathcal{P}(2)\,u&=&\left[ b_{1,0,1}(x)\emph{\emph{D}}_{0,1}+b_{1,1,1}(x)\emph{\emph{D}}_{1,1} +b_{1,1,2}(x)\emph{\emph{D}}_{1,2} \right]\nonumber\\
&\times&
\left[  b_{2,0,1}(x)\emph{\emph{D}}_{0,1}+b_{2,1,1}(x)\emph{\emph{D}}_{1,1} +b_{2,1,2}(x)\emph{\emph{D}}_{1,2}\right]u\nonumber\\
&=&\left[ b_{1,0,1}b_{2,0,1}+b_{1,1,1}\emph{\emph{D}}_{1,1}(b_{2,0,1})+b_{1,1,2}\emph{\emph{D}}_{1,2}(b_{2,0,1}) \right]\,u\nonumber\\
&+&\left[  b_{1,0,1}b_{2,1,1}+b_{1,1,1}b_{2,0,1}+b_{1,1,1}\emph{\emph{D}}_{1,1}(b_{2,1,1})+b_{1,1,2}\emph{\emph{D}}_{1,2}(b_{2,1,1})\right]\,u_{x^1}\nonumber\\
&+&\left[  b_{1,0,1}b_{2,1,2}+b_{1,1,2}b_{2,0,1}+b_{1,1,1}\emph{\emph{D}}_{1,1}(b_{2,1,2})+b_{1,1,2}\emph{\emph{D}}_{1,2}(b_{2,1,2})\right]\,u_{x^2}\nonumber\\
&+&b_{1,1,1}b_{2,1,1}\,u_{2x^1}+\left[ b_{1,1,2}b_{2,1,1}+b_{1,1,1}b_{2,1,2}  \right]\,u_{x^1x^2}+     b_{1,1,2}b_{2,1,2}  \,u_{2x^2}.
\end{eqnarray}
Identifying (\ref{eqaa11}) with (\ref{eqaa12}) leads to the following
\begin{proposition}
A necessary and sufficient condition for the differential operator $\mathcal{P}(2) $ defined by (\ref{eqaa9})  be decomposed into
the form (\ref{eqaa10}) is:
\begin{eqnarray}
g_{2,1} &=& b_{1,1,1}b_{2,1,1},
\label{eqaa13}\\
g_{2,2}+ g_{2,3}&=& b_{1,1,2}b_{2,1,1}+b_{1,1,1}b_{2,1,2},\label{eqaa14}\\
g_{2,4} &=& b_{1,1,2}b_{2,1,2},\label{eqaa15}\\
g_{1,1} &=&b_{1,0,1}b_{2,1,1}+b_{1,1,1}b_{2,0,1}+\mathcal{L}(b_{2,1,1}),\label{eqaa16}\\
g_{1,2} &=&b_{1,0,1}b_{2,1,2}+b_{1,1,2}b_{2,0,1}+\mathcal{L}(b_{2,1,2}),\label{eqaa17}\\
g_{0,1} &=&b_{1,0,1}b_{2,0,1}+\mathcal{L}(b_{2,0,2}),\label{eqaa18}
\end{eqnarray}
where $\mathcal{L}=b_{1,1,1}\emph{\emph{D}}_{1,1}+b_{1,1,2}\emph{\emph{D}}_{1,2}.$
\end{proposition}

Propose an approach to solve system (\ref{eqaa13})-(\ref{eqaa18}).
Assume that at least one of the functions $g_{2,1}$ and $g_{2,4}$ does not vanish on $\Lambda,$ says $g_{2,1}.$
It is always possible to find two nonzero functions on $\Lambda,$ namely $b_{1,1,1}$ and $b_{2,1,1}$ which satisfy
(\ref{eqaa13}). Substituting $X_1=b_{1,1,2}$ and $X_2=b_{2,1,2}$ into (\ref{eqaa14}) yields
\begin{equation}\label{eqaa19}
X_1=\frac{1}{b_{2,1,1}}\left(g_{2,2}+g_{2,3}-b_{1,1,1}X_2   \right).
\end{equation}
The substitution of (\ref{eqaa19}) into (\ref{eqaa15}) shows that $b_{2,1,2}$ is a solution
of the second degree algebraic equation
\begin{equation}\label{eqaa20}
\frac{b_{1,1,1}}{g_{2,1}}X^2_2-\frac{g_{2,2}+ g_{2,3}}{g_{2,1}}X_2+\frac{g_{2,4}}{b_{1,1,1}}=0.
\end{equation}
The discriminant of equation (\ref{eqaa20}) is
\begin{equation}\label{eqaa21}
\Delta=\left(g_{2,2}+ g_{2,3}\right)^2-4g_{2,1}g_{2,4}=\left(b_{1,1,2}b_{2,1,1}-b_{1,1,1}b_{2,1,2}\right)^2  \geq 0.
\end{equation}
If $\Delta >0,$ then the substitution of $Y=b_{1,0,1}$ and $Z=b_{2,0,1}$ into (\ref{eqaa16}) and (\ref{eqaa17})
implies that the decomposition (\ref{eqaa10})  is possible if
  the unique solution to the following algebraic system  in $Y$ and $Z$
\begin{eqnarray}\label{eqaa22}
g_{1,1}-\mathcal{L}(b_{2,1,1})&=& b_{2,1,1}Y+b_{1,1,1}Z\nonumber\\
g_{1,2}-\mathcal{L}(b_{2,1,2})&=& b_{2,1,2}Y+b_{1,1,2}Z
\end{eqnarray}
satisfies (\ref{eqaa18}). Indeed, the determinant of the system (\ref{eqaa22}) is
$$b_{1,1,2}b_{2,1,1}-b_{1,1,1}b_{2,1,2}=\pm\sqrt{\Delta}\neq 0.$$
If $\Delta = 0,$  then the substitution of $Y=b_{1,0,1}$ and $Z=b_{2,0,1}$ into (\ref{eqaa16}) yields
\begin{equation}\label{eqaa23}
Y=\frac{1}{b_{2,1,1}}\left[ g_{1,1}-\mathcal{L}(b_{2,1,1})-b_{1,1,1}Z    \right].
\end{equation}
Then, the substitution of (\ref{eqaa23}) into (\ref{eqaa18}) implies that the decomposition (\ref{eqaa10})  is strongly related
to the existence of a solution to the following  first order quasi-linear partial differential equation in $Z$
\begin{equation}\label{eqaa24}
\mathcal{L}(Z)-\frac{b_{1,1,1}}{b_{2,1,1}}Z^2+\frac{g_{1,1}-\mathcal{L}(b_{2,1,1})}{b_{2,1,1}}Z-g_{0,1}=0
\end{equation}
which satisfies (\ref{eqaa17}).

\subsection{Factorizations of systems of linear  differential equations}
The previous analysis is now made for systems of linear differential equations.
\subsubsection{General considerations}

Let $\Lambda$ be an open subset of $\mathbb{R}^n.$
 Examine now the factorization process for systems of $s$-th order, $(s\geq 2),$ linear
 differential equations with $n$ independent variables $x=\left(x^1, \cdots ,x^n\right)$
 and $m\geq 2$ dependent variables $u=\,^{t}\left(u^1, \cdots ,u^m\right),$ $u=u(x)$
 whose associated matrix operator, $\mathcal{M}(s),$ is of the form
 \begin{equation}\label{eqb1}
\mathcal{M}(s)=\left[  \mathcal{R}_{p,q}\left(s_{p,q}\right)   \right]_{1\leq p,q\leq m};
 \end{equation}
the  $\mathcal{R}_{p,q}\left(s_{p,q}\right)$ are $s_{p,q}$-th order linear differential operators
 \begin{equation}\label{eqb2}
\mathcal{R}_{p,q}\left(s_{p,q}\right)=\sum_{k=0}^{s_{p,q}}\sum_{h=1}^{p_k}f_{p,q,k,h}(x)\emph{\emph{D}}_{k,h},
 \end{equation}
where $f_{p,q,k,h}\in \mathcal{C}(\Lambda,\mathbb{R}),$ $s_{p,q}=s-1+\delta_{p,q},$ $\delta_{p,p}=1$ and $\delta_{p,q}= 0$ if $p\neq q.$  \\
Let $\Lambda$ and $\Lambda_0$ be two open subsets of $\mathbb{R}^n$ such that $\Lambda_0 \subset \Lambda.$ The matrix operator $\mathcal{M}(s)$ acts on a vector valued function $u=\,^{t}\left(u^1, \cdots ,u^m\right) \in \mathcal{C}^s(\Lambda_0,\mathbb{R}^m)$ as follows
$$
\mathcal{M}(s) \,u=\left[  \mathcal{R}_{p,q}\left(s_{p,q}\right)   \right]_{1\leq p,q\leq m}\,u=\left[  \sum_{q=1}^{m} \mathcal{R}_{p,q}\left(s_{p,q}\right)  \,u^q \right]_{1\leq p\leq m}.
$$
The method of factorization consists in seeking a decomposition of the matrix   $\mathcal{M}(s)$ under the following form
 \begin{equation}\label{eqb3}
\mathcal{M}(s)=\prod_{i=1}^{l}\mathcal{N}_i(s_i)
 \end{equation}
where
\begin{equation}\label{eqb4}
\mathcal{N}_i(s_i)=\left[  \mathcal{T}_{i,p,q}\left(s_{i,p,q}\right)   \right]_{1\leq p,q\leq m}
 \end{equation}
  and
\begin{equation}\label{eqb5}
\mathcal{T}_{i,p,q}\left(s_{i,p,q}\right)=\sum_{k=0}^{s_{i,p,q}}\sum_{h=1}^{p_k}a_{i,p,q,k,h}(x)\emph{\emph{D}}_{k,h},
 \end{equation}
with $\sum_{i=1}^l s_i=s,$   $s_{i,p,q}=s_i-1+\delta_{p,q},\,$
$a_{1,p,q,k,h} \in \mathcal{C}(\Lambda, \mathbb{R})$
 and  $a_{i,p,q,k,h} \in \mathcal{C}^{\sum_{j=1}^{i-1}s_{i,p,q}}(\Lambda, \mathbb{R}),$
$i=2,3, \cdots  ,l.$
\begin{proposition}
Let $\mathcal{M}(s)$ be a matrix of differential operators defined by (\ref{eqb1})   which can be decomposed into the form
 (\ref{eqb3}). If the function $u_0=\,^{t}\left(u^1_0, \cdots ,u^m_0\right)$  satisfies
\begin{equation}\label{eqb6}
\mathcal{N}_l(s_l)u_0=0,
\end{equation}
and $u_j=\,^{t}\left(u^1_j, \cdots ,u^m_j\right),$ $j=1,2,\cdots  ,l-1$ are solutions of the system
\begin{eqnarray}
 \label{eqb7}
  \prod_{k=l-j+1}^{l}\mathcal{N}_{k}(s_k)u_j&=& v_{j},\;j=1,\,2,\ldots,\,l-1,
\end{eqnarray}
where $v_j=\,^{t}\left(v^1_j, \cdots ,v^m_j\right),$ $j=1,\,2,\ldots,\,l-1,$ are solutions of
\begin{eqnarray}
\label{eqb8}
\prod_{i=1}^{l-j}\mathcal{N}_i(s_i)v_j&=&0,
\end{eqnarray}
then $u_0,$ $u_1,\,\ldots,\,u_{l-1}$ are $l$ particular solutions of the equation
$\mathcal{M}(s)\,u=0.$
\end{proposition}
\begin{proof}
The proof is similar to that of the \textrm{Proposition} \ref{pn1}.
\end{proof}

Expanding (\ref{eqb3}) leads to the relations between the unknown functions   $a_{i,p,q,k,h}$ of $\mathcal{N}_i(s_i)$
and the known functions $f_{p,q,k,h}$ of $\mathcal{M}(s).$

As matter of clarity, in the sequel we explicitly derive necessary and sufficient conditions for the factorization of systems of second order linear ordinary
and partial differential operators with two  independent variables.

\subsubsection{Necessary and sufficient conditions for the factorization of systems of second order linear ODEs}

Let $\Lambda$ and $\Lambda_0$ be two open subsets of $\mathbb{R}$ such that $\Lambda_0 \subset \Lambda.$
Consider the matrix operator
\begin{equation}\label{eqbb1}
\mathcal{M}(2)=\left[  \mathcal{R}_{p,q}  \right]_{1\leq p,q\leq m},
 \end{equation}
where
\begin{equation}\label{eqbb2}
\mathcal{R}_{p,p}=\sum_{k=0}^{2}\sum_{h=1}^{p_k}f_{p,q,k,h}(x)\emph{\emph{D}}_{k,h} =f_{p,p,0,1}\emph{\emph{D}}_{0,1}+f_{p,p,1,1}\emph{\emph{D}}_{1,1}+f_{p,p,2,1}\emph{\emph{D}}_{2,1}
 \end{equation}
and for $p \neq q$
\begin{equation}\label{eqbb3}
\mathcal{R}_{p,q}=\sum_{k=0}^{1}\sum_{h=1}^{p_k}f_{p,q,k,h}(x)\emph{\emph{D}}_{k,h} =f_{p,q,0,1}\emph{\emph{D}}_{0,1}+f_{p,q,1,1}\emph{\emph{D}}_{1,1}
 \end{equation}
with $f_{p,q,k,h}\in \mathcal{C}(\Lambda,\mathbb{R}),$ $x=x^1.$ Write $\mathcal{M}(2)$ in the form
\begin{equation}\label{eqbb4}
\mathcal{M}(2)=\mathcal{N}_1(1)\cdot \mathcal{N}_2(1),
 \end{equation}
where
\begin{equation}\label{eqbb5}
\mathcal{N}_i(1)=\left[  \mathcal{T}_{i,p,q}  \right]_{1\leq p,q\leq m}
 \end{equation}
with
\begin{equation}\label{eqbb6}
\mathcal{T}_{i,p,p}=\sum_{k=0}^{1}\sum_{h=1}^{p_k}a_{i,p,p,k,h}(x)\emph{\emph{D}}_{k,h} =a_{i,p,p,0,1}\emph{\emph{D}}_{0,1}+a_{i,p,p,1,1}\emph{\emph{D}}_{1,1}
 \end{equation}
and for $p \neq q$
\begin{equation}\label{eqbb7}
\mathcal{T}_{i,p,q} =a_{i,p,q,0,1}\emph{\emph{D}}_{0,1},
 \end{equation}
$a_{1,p,q,k,h}\in \mathcal{C}(\Lambda,\mathbb{R})$ and $a_{2,p,q,k,h}\in \mathcal{C}^1(\Lambda,\mathbb{R}).$
Let $u =\,^{t}\left(u^1, \cdots ,u^m\right) \in \mathcal{C}^2(\Lambda_0,\mathbb{R}).$ Then we have
\begin{equation}\label{eqbb8}
\mathcal{M}(2)\,u = \left[  \mathcal{R}_{p,q}  \right]_{1\leq p,q\leq m}\,u=\left[ \sum_{q=1}^m \mathcal{R}_{p,q}\,u^q  \right]_{1\leq p\leq m}
\end{equation}
where
$$  \mathcal{R}_{p,p}\,u^p= f_{p,p,0,1}\,u^p+f_{p,p,1,1}\,u^p_x+f_{p,p,2,1}\,u^p_{2x} $$
and for $p \neq q$
$$ \mathcal{R}_{p,q}\,u^q =f_{p,q,0,1}\,u^q+f_{p,q,1,1}\,u^q_x.  $$
On the other hand, after expansion of (\ref{eqbb4}), we have
\begin{equation}\label{eqbb9}
\mathcal{M}(2)\,u = \left[  \widetilde{\mathcal{R}}_{p,q}  \right]_{1\leq p,q\leq m}\,u=\left[ \sum_{q=1}^m \widetilde{\mathcal{R}}_{p,q}\,u^q  \right]_{1\leq p\leq m},
\end{equation}
where
\begin{eqnarray}
\widetilde{\mathcal{R}}_{p,p}\,u^p&=&\left[\sum_{l=1}^m a_{1,p,l,0,1}a_{2,l,p,0,1}+a_{1,p,p,1,1}\emph{\emph{D}}_{1,1}(a_{2,p,p,0,1}) \right]u^p+a_{1,p,p,1,1}a_{2,p,p,1,1}\,u^p_{2x}\nonumber\\
&+& \left[a_{1,p,p,0,1}a_{2,p,p,1,1}+a_{1,p,p,1,1}a_{2,p,p,0,1}+a_{1,p,p,1,1}\emph{\emph{D}}_{1,1}(a_{2,p,p,1,1}) \right]u^p_x\nonumber
\end{eqnarray}
and for $p \neq q$
\begin{eqnarray}
\widetilde{\mathcal{R}}_{p,q}\,u^q&=&\left[a_{1,p,p,1,1}a_{2,p,q,0,1}+a_{1,p,q,0,1}a_{2,q,q,1,1} \right]u^q_x\nonumber\\
&+&\left[\sum_{l=1}^m a_{1,p,l,0,1}a_{2,l,q,0,1}+a_{1,p,p,1,1}\emph{\emph{D}}_{1,1}(a_{2,p,q,0,1}) \right]u^q.\nonumber
\end{eqnarray}
Identifying (\ref{eqbb8}) with (\ref{eqbb9}) yields
\begin{proposition}
A necessary and sufficient condition for the differential operator $\mathcal{M}(2) $ defined by (\ref{eqbb1})  be decomposed into
the form (\ref{eqbb4}) is:
\begin{eqnarray}
f_{p,p,0,1}&=&\sum_{l=1}^m a_{1,p,l,0,1}a_{2,l,p,0,1}+a_{1,p,p,1,1}\emph{\emph{D}}_{1,1}(a_{2,p,p,0,1}),\\
f_{p,p,1,1}&=& a_{1,p,p,0,1}a_{2,p,p,1,1}+a_{1,p,p,1,1}a_{2,p,p,0,1}+a_{1,p,p,1,1}\emph{\emph{D}}_{1,1}(a_{2,p,p,1,1}),\\
f_{p,p,2,1}    &=&a_{1,p,p,1,1}a_{2,p,p,1,1}
\end{eqnarray}
and for $p \neq q$
\begin{eqnarray}
f_{p,q,0,1}&=&\sum_{l=1}^m a_{1,p,l,0,1}a_{2,l,q,0,1}+a_{1,p,p,1,1}\emph{\emph{D}}_{1,1}(a_{2,p,q,0,1}),\\
f_{p,q,1,1}   &=& a_{1,p,p,1,1}a_{2,p,q,0,1}+a_{1,p,q,0,1}a_{2,q,q,1,1}.
\end{eqnarray}
\end{proposition}

\subsubsection{Necessary and sufficient conditions for the factorization of systems of second order linear PDEs with two independent variables}

Let $\Lambda$ and $\Lambda_0$ be two open subsets of $\mathbb{R}^2$ such that $\Lambda_0 \subset \Lambda.$
Consider the matrix operator
\begin{equation}\label{eqbbb1}
\mathcal{M}(2)=\left[  \mathcal{R}_{p,q}  \right]_{1\leq p,q\leq m},
 \end{equation}
where
\begin{eqnarray}\label{eqbbb2}
\mathcal{R}_{p,p}&=&\sum_{k=0}^{2}\sum_{h=1}^{p_k}f_{p,q,k,h}(x)\emph{\emph{D}}_{k,h} \nonumber\\ &=&f_{p,p,0,1}\emph{\emph{D}}_{0,1}+f_{p,p,1,1}\emph{\emph{D}}_{1,1}+f_{p,p,1,2}\emph{\emph{D}}_{1,2}\nonumber\\
&+&f_{p,p,2,1}\emph{\emph{D}}_{2,1}
+f_{p,p,2,2}\emph{\emph{D}}_{2,2}+f_{p,p,2,3}\emph{\emph{D}}_{2,3}+f_{p,p,2,4}\emph{\emph{D}}_{2,4}
 \end{eqnarray}
and for $p \neq q$
\begin{equation}\label{eqbbb3}
\mathcal{R}_{p,q}=\sum_{k=0}^{1}\sum_{h=1}^{p_k}f_{p,q,k,h}(x)\emph{\emph{D}}_{k,h} =f_{p,q,0,1}\emph{\emph{D}}_{0,1}+f_{p,q,1,1}\emph{\emph{D}}_{1,1}+f_{p,q,1,2}\emph{\emph{D}}_{1,2}
 \end{equation}
with $f_{p,q,k,h}\in \mathcal{C}(\Lambda,\mathbb{R}),$ $x= \left(x^1,x^2\right).$ Write $\mathcal{M}(2)$ in the form
\begin{equation}\label{eqbbb4}
\mathcal{M}(2)=\mathcal{N}_1(1)\cdot \mathcal{N}_2(1),
 \end{equation}
where
\begin{equation}\label{eqbbb5}
\mathcal{N}_i(1)=\left[  \mathcal{T}_{i,p,q}  \right]_{1\leq p,q\leq m}
 \end{equation}
with
\begin{equation}\label{eqbbb6}
\mathcal{T}_{i,p,p}=\sum_{k=0}^{1}\sum_{h=1}^{p_k}a_{i,p,p,k,h}(x)\emph{\emph{D}}_{k,h} =a_{i,p,p,0,1}\emph{\emph{D}}_{0,1}+a_{i,p,p,1,1}\emph{\emph{D}}_{1,1}+a_{i,p,p,1,2}\emph{\emph{D}}_{1,2}
 \end{equation}
and for $p \neq q$
\begin{equation}\label{eqbbb7}
\mathcal{T}_{i,p,q} =a_{i,p,q,0,1}\emph{\emph{D}}_{0,1},
 \end{equation}
$a_{1,p,q,k,h}\in \mathcal{C}(\Lambda,\mathbb{R})$ and $a_{2,p,q,k,h}\in \mathcal{C}^1(\Lambda,\mathbb{R}).$
Let $u =\,^{t}\left(u^1, \cdots ,u^m\right) \in \mathcal{C}^2(\Lambda_0,\mathbb{R}).$ Then we have
\begin{equation}\label{eqbbb8}
\mathcal{M}(2)\,u = \left[  \mathcal{R}_{p,q}  \right]_{1\leq p,q\leq m}\,u=\left[ \sum_{q=1}^m \mathcal{R}_{p,q}\,u^q  \right]_{1\leq p\leq m},
\end{equation}
where
\begin{eqnarray}
  \mathcal{R}_{p,p}\,u^p&=& f_{p,p,0,1}\,u^p+f_{p,p,1,1}\,u^p_{x^1}+f_{p,p,1,2}\,u^p_{x^2}\nonumber\\
&+&f_{p,p,2,1}\,u^p_{2x^1}
+\left(f_{p,p,2,2}+f_{p,p,2,3}\right)\,u^p_{x^1x^2}+f_{p,p,2,4}\,u^p_{2x^2} \nonumber
\end{eqnarray}
and for $p \neq q$
$$ \mathcal{R}_{p,q}\,u^q =f_{p,q,0,1}\,u^q+f_{p,q,1,1}\,u^q_{x^1}+f_{p,q,1,2}\,u^q_{x^2}.  $$
On the other hand, after expansion of (\ref{eqbbb4}), we have
\begin{equation}\label{eqbbb9}
\mathcal{M}(2)\,u = \left[  \widetilde{\mathcal{R}}_{p,q}  \right]_{1\leq p,q\leq m}\,u=\left[ \sum_{q=1}^m \widetilde{\mathcal{R}}_{p,q}\,u^q  \right]_{1\leq p\leq m},
\end{equation}
where
\begin{eqnarray}
\widetilde{\mathcal{R}}_{p,p}\,u^p&=&a_{1,p,p,1,1}a_{2,p,p,1,1}\,u^p_{2x^1}+\left(a_{1,p,p,1,2}a_{2,p,p,1,1}+a_{1,p,p,1,1}a_{2,p,p,1,2}\right)\,u^p_{x^1x^2}
\nonumber\\
&+& \left[a_{1,p,p,0,1}a_{2,p,p,1,1}+a_{1,p,p,1,1}a_{2,p,p,0,1}+\mathcal{L}_p(a_{2,p,p,1,1}) \right]u^p_{x^1}\nonumber\\
&+&\left[a_{1,p,p,0,1}a_{2,p,p,1,2}+a_{1,p,p,1,2}a_{2,p,p,0,1}+\mathcal{L}_p(a_{2,p,p,1,2}) \right]u^p_{x^2}\nonumber\\
&+&\left[\sum_{l=1}^m a_{1,p,l,0,1}a_{2,l,p,0,1}+\mathcal{L}_p(a_{2,p,p,0,1}) \right]u^p+a_{1,p,p,1,2}a_{2,p,p,1,2}\,u^p_{2x^2}\nonumber
\end{eqnarray}
and for $p \neq q$
\begin{eqnarray}
\widetilde{\mathcal{R}}_{p,q}\,u^q&=&\left[a_{1,p,p,1,1}a_{2,p,q,0,1}+a_{1,p,q,0,1}a_{2,q,q,1,1} \right]u^q_{x^1}\nonumber\\
&+&\left[a_{1,p,p,1,2}a_{2,p,q,0,1}+a_{1,p,q,0,1}a_{2,q,q,1,2} \right]u^q_{x^2}\nonumber\\
&+&\left[\sum_{l=1}^m a_{1,p,l,0,1}a_{2,l,q,0,1}+\mathcal{L}_p(a_{2,p,q,0,1}) \right]u^q,\nonumber
\end{eqnarray}
where $\mathcal{L}_p=a_{1,p,p,1,1}\emph{\emph{D}}_{1,1}+a_{1,p,p,1,2}\emph{\emph{D}}_{1,2}.$
From the Identification of (\ref{eqbbb8}) with (\ref{eqbbb9}) results
\begin{proposition}
A necessary and sufficient condition for the differential operator $\mathcal{M}(2) $ defined by (\ref{eqbbb1})  be decomposed into
the form (\ref{eqbbb4}) is:
\begin{eqnarray}
f_{p,p,0,1}&=&\sum_{l=1}^m a_{1,p,l,0,1}a_{2,l,p,0,1}+\mathcal{L}_p(a_{2,p,p,0,1}),\\
f_{p,p,1,1}&=& a_{1,p,p,0,1}a_{2,p,p,1,1}+a_{1,p,p,1,1}a_{2,p,p,0,1}+\mathcal{L}_p(a_{2,p,p,1,1}),\\
f_{p,p,1,2}&=& a_{1,p,p,0,1}a_{2,p,p,1,2}+a_{1,p,p,1,2}a_{2,p,p,0,1}+\mathcal{L}_p(a_{2,p,p,1,2}),\\
f_{p,p,2,1}    &=&a_{1,p,p,1,1}a_{2,p,p,1,1},\\
f_{p,p,2,2}+f_{p,p,2,3}    &=&a_{1,p,p,1,2}a_{2,p,p,1,1}+a_{1,p,p,1,1}a_{2,p,p,1,2},\\
f_{p,p,2,4}    &=&a_{1,p,p,1,2}a_{2,p,p,1,2}
\end{eqnarray}
and for $p \neq q$
\begin{eqnarray}
f_{p,q,0,1}&=&\sum_{l=1}^m a_{1,p,l,0,1}a_{2,l,q,0,1}+\mathcal{L}_p(a_{2,p,q,0,1}),\\
f_{p,q,1,1}   &=& a_{1,p,p,1,1}a_{2,p,q,0,1}+a_{1,p,q,0,1}a_{2,q,q,1,1},\\
f_{p,q,1,2}   &=& a_{1,p,p,1,2}a_{2,p,q,0,1}+a_{1,p,q,0,1}a_{2,q,q,1,2}.
\end{eqnarray}
\end{proposition}

\section{Nonlinear  differential operators}
In this section, we investigate  the factorization of nonlinear differential operators.

\subsection{Factorizations of nonlinear  differential equations}
We start with general considerations and then deduce the main results on conditions of factorization.
\subsubsection{General setting and results}

Let $s\geq 2$ be a positive integer, $\Lambda$ be an open subset of $\mathbb{R}^n$
and $\Omega$ an open subset of $\mathbb{R}.$
Let
\begin{equation}\label{nleqa2}
\mathcal{P}(s)=\sum_{k=0}^{s}\sum_{h=1}^{p_k}g_{k,h}(x,\cdot)\emph{\emph{D}}_{k,h}
\end{equation}
be a nonlinear differential operator of order $s,$ where $g_{k,h}\in \mathcal{C}(\Lambda\times\Omega,\mathbb{R}).$
The operator $\mathcal{P}(s)$ acts on a function $u\in \mathcal{C}^{s}(\Lambda,\Omega)$ as follows
\begin{equation}\label{nleqa3}
\mathcal{P}(s)\,u=\sum_{k=0}^{s}\sum_{h=1}^{p_k}g_{k,h}(x,u)\emph{\emph{D}}_{k,h}\,u.
\end{equation}
The method of factorization consists in seeking a decomposition of the differential operator (\ref{nleqa2})
in the following form
\begin{equation}\label{nleqa4}
\mathcal{P}(s)=\prod_{i=1}^{l}\mathcal{Q}_i(s_i)
\end{equation}
with $\sum_{i=1}^{l}s_i=s$ and
\begin{equation}\label{nleqa5}
\mathcal{Q}_i(s_i)=\sum_{k=0}^{s_i}\sum_{h=1}^{p_k}b_{i,k,h}(x,\cdot)\emph{\emph{D}}_{k,h},
\end{equation}
where $b_{1,k,h}\in \mathcal{C}(\Lambda\times\Omega,\mathbb{R})$  and  $b_{i,k,h}\in \mathcal{C}^{\sum_{j=1}^{i-1}s_j}(\Lambda\times\Omega,\mathbb{R}),$
$i=2,3, \cdots,l.$\\
Expanding (\ref{nleqa4}) leads to the relations between unknown
functions $b_{i,k,h}$ of the differential operators
$\mathcal{Q}_i(s_i)$ and the known functions $g_{k,h}$ of
the original differential operator $\mathcal{P}(s).$

\subsubsection{Necessary and sufficient conditions for the factorization of second order nonlinear ODEs}

Let $\Omega,\Lambda$ and $\Lambda_0$ be three open subsets of $\mathbb{R}$ such that $\Lambda_0 \subset  \Lambda.$
Consider the second order nonlinear ordinary differential operator
\begin{eqnarray}\label{nleqa9}
\mathcal{P}(2)&=&\sum_{k=0}^{2}\sum_{h=1}^{p_k}g_{k,h}(x,\cdot)\emph{\emph{D}}_{k,h}\nonumber\\
&=& g_{0,1}(x,\cdot)\emph{\emph{D}}_{0,1}+g_{1,1}(x,\cdot)\emph{\emph{D}}_{1,1}
+g_{2,1}(x,\cdot)\emph{\emph{D}}_{2,1},
\end{eqnarray}
where $g_{k,h}\in \mathcal{C}(\Lambda\times\Omega,\mathbb{R})$ and  $x=x^1.$
Write $\mathcal{P}(2)$ in the form
\begin{eqnarray}\label{nleqa10}
\mathcal{P}(2)&=&\mathcal{Q}_1(1)\cdot\mathcal{Q}_2(1)\nonumber\\
&=&\left[\sum_{k=0}^{1}\sum_{h=1}^{p_k}b_{1,k,h}(x,\cdot)\emph{\emph{D}}_{k,h}\right]
\left[\sum_{k=0}^{1}\sum_{h=1}^{p_k}b_{2,k,h}(x,\cdot)\emph{\emph{D}}_{k,h}\right]\nonumber\\
&=&\left[ b_{1,0,1}(x,\cdot)\emph{\emph{D}}_{0,1}+b_{1,1,1}(x,\cdot)\emph{\emph{D}}_{1,1}  \right]
\left[  b_{2,0,1}(x,\cdot)\emph{\emph{D}}_{0,1}+b_{2,1,1}(x,\cdot)\emph{\emph{D}}_{1,1} \right],
\end{eqnarray}
where $b_{1,k,h}\in \mathcal{C}(\Lambda\times\Omega,\mathbb{R})$  and  $b_{2,k,h}\in \mathcal{C}^{1}(\Lambda\times\Omega,\mathbb{R}).$
Let $u\in \mathcal{C}^2(\Lambda_0,\Omega).$ Then we have
\begin{eqnarray}\label{nleqa11}
\mathcal{P}(2)\,u
&=& g_{0,1}(x,u)\,u+g_{1,1}(x,u)\,u_x
+g_{2,1}(x,u)\,u_{2x}
\end{eqnarray}
and after expansion
\begin{eqnarray}\label{nleqa12}
\mathcal{P}(2)\,u&=&\left[ b_{1,0,1}(x,\cdot)\emph{\emph{D}}_{0,1}+b_{1,1,1}(x,\cdot)\emph{\emph{D}}_{1,1}  \right]
\left[  b_{2,0,1}(x,\cdot)\emph{\emph{D}}_{0,1}+b_{2,1,1}(x,\cdot)\emph{\emph{D}}_{1,1} \right]u\nonumber\\
&=&
\left[ b_{1,0,1}b_{2,1,1}+b_{1,1,1}b_{2,0,1}+b_{1,1,1}\emph{\emph{D}}_{1,1}(b_{2,1,1}) +b_{1,1,1}\emph{\emph{D}}_{1,2}(b_{2,0,1}) u \right]u_x\\
&+&b_{1,1,1}\emph{\emph{D}}_{1,2}(b_{2,1,1})u_x^2+\left[ b_{1,0,1}b_{2,0,1}+b_{1,1,1}\emph{\emph{D}}_{1,1}(b_{2,0,1})  \right]u+b_{1,1,1}b_{2,1,1}\,u_{2x}.\nonumber
\end{eqnarray}
Identifying (\ref{nleqa11}) with (\ref{nleqa12}) furnishes
\begin{proposition}
A necessary and  sufficient condition for the differential operator $\mathcal{P}(2) $ defined by (\ref{nleqa9})  be decomposed into
the form (\ref{nleqa10}) is:
\begin{eqnarray}
g_{2,1} &=& b_{1,1,1}b_{2,1,1},
\label{nleqa13}\\
g_{1,1} &=& b_{1,0,1}b_{2,1,1}+b_{1,1,1}b_{2,0,1}+b_{1,1,1}\emph{\emph{D}}_{1,1}(b_{2,1,1}) +b_{1,1,1}\emph{\emph{D}}_{1,2}(b_{2,0,1}) u ,\label{nleqa14}\\
0&=&b_{1,1,1}\emph{\emph{D}}_{1,2}(b_{2,1,1}),\label{nleqa15}\\
g_{0,1} &=& b_{1,0,1}b_{2,0,1}+b_{1,1,1}\emph{\emph{D}}_{1,1}(b_{2,0,1}).\label{nleqa16}
\end{eqnarray}
\end{proposition}

\subsubsection{Necessary and sufficient conditions for the factorization of second order nonlinear PDEs with two independent variables}

Let $\Lambda$ and $\Lambda_0$ be two open subsets of $\mathbb{R}^2$ such that $\Lambda_0 \subset  \Lambda.$
Let $\Omega$ be an open subset of $\mathbb{R}.$
Consider the second order nonlinear partial differential operator
\begin{eqnarray}\label{nleqaa9}
\mathcal{P}(2)&=&\sum_{k=0}^{2}\sum_{h=1}^{p_k}g_{k,h}(x,\cdot)\emph{\emph{D}}_{k,h}\nonumber\\
&=& g_{0,1}(x,\cdot)\emph{\emph{D}}_{0,1}+g_{1,1}(x,\cdot)\emph{\emph{D}}_{1,1}+g_{1,2}(x,\cdot)\emph{\emph{D}}_{1,2}
+g_{2,1}(x,\cdot)\emph{\emph{D}}_{2,1}\nonumber\\
&+&g_{2,2}(x,\cdot)\emph{\emph{D}}_{2,2}+g_{2,3}(x,\cdot)\emph{\emph{D}}_{2,3}+g_{2,4}(x,\cdot)\emph{\emph{D}}_{2,4},
\end{eqnarray}
where $g_{k,h}\in \mathcal{C}(\Lambda\times\Omega,\mathbb{R})$ and $x=\left(x^1,x^2\right).$
Write $\mathcal{P}(2)$ in the form
\begin{eqnarray}\label{nleqaa10}
\mathcal{P}(2)&=&\mathcal{Q}_1(1)\cdot\mathcal{Q}_2(1)\nonumber\\
&=&\left[\sum_{k=0}^{1}\sum_{h=1}^{p_k}b_{1,k,h}(x,\cdot)\emph{\emph{D}}_{k,h}\right]
\left[\sum_{k=0}^{1}\sum_{h=1}^{p_k}b_{2,k,h}(x,\cdot)\emph{\emph{D}}_{k,h}\right]\nonumber\\
&=&\left[ b_{1,0,1}(x,\cdot)\emph{\emph{D}}_{0,1}+b_{1,1,1}(x,\cdot)\emph{\emph{D}}_{1,1} +b_{1,1,2}(x,\cdot)\emph{\emph{D}}_{1,2} \right]\nonumber\\
&\times& \left[  b_{2,0,1}(x,\cdot)\emph{\emph{D}}_{0,1}+b_{2,1,1}(x,\cdot)\emph{\emph{D}}_{1,1} +b_{2,1,2}(x,\cdot)\emph{\emph{D}}_{1,2}\right],
\end{eqnarray}
where $b_{1,k,h}\in \mathcal{C}(\Lambda\times\Omega,\mathbb{R})$  and  $b_{2,k,h}\in \mathcal{C}^{1}(\Lambda\times\Omega,\mathbb{R}).$
Let $u\in \mathcal{C}^2(\Lambda_0,\Omega).$ Then we have
\begin{eqnarray}\label{nleqaa11}
\mathcal{P}(2)\,u
&=& g_{0,1}(x,u)\,u+g_{1,1}(x,u)\,u_{x^1}+g_{1,2}(x,u)\,u_{x^2}
+g_{2,1}(x,u)\,u_{2x^1}\nonumber\\
&+&\left(g_{2,2}(x,u)+g_{2,3}(x,u)\right)\,u_{x^1x^2}+g_{2,4}(x,u)\,u_{2x^2}
\end{eqnarray}
and after expansion
\begin{eqnarray}\label{nleqaa12}
\mathcal{P}(2)\,u&=&\left[ b_{1,0,1}(x,\cdot)\emph{\emph{D}}_{0,1}+b_{1,1,1}(x,\cdot)\emph{\emph{D}}_{1,1} +b_{1,1,2}(x,\cdot)\emph{\emph{D}}_{1,2} \right]\nonumber\\
&\times& \left[  b_{2,0,1}(x,\cdot)\emph{\emph{D}}_{0,1}+b_{2,1,1}(x,\cdot)\emph{\emph{D}}_{1,1} +b_{2,1,2}(x,\cdot)\emph{\emph{D}}_{1,2}\right]u\nonumber\\
&=&\left[ b_{1,0,1}b_{2,0,1}+b_{1,1,1}\emph{\emph{D}}_{1,1}(b_{2,0,1})+b_{1,1,2}\emph{\emph{D}}_{1,2}(b_{2,0,1}) \right]\,u   \nonumber\\
&+&\left[  b_{1,0,1}b_{2,1,1}+b_{1,1,1}b_{2,0,1}+b_{1,1,1}\emph{\emph{D}}_{1,1}(b_{2,1,1})+b_{1,1,2}\emph{\emph{D}}_{1,2}(b_{2,1,1})\right.\nonumber\\
&+&\left.b_{1,1,1}\emph{\emph{D}}_{1,3}(b_{2,0,1})\,u\right]\,u_{x^1}
+\left[  b_{1,0,1}b_{2,1,2}+b_{1,1,2}b_{2,0,1}+b_{1,1,1}\emph{\emph{D}}_{1,1}(b_{2,1,2})\right.\nonumber\\
&+&\left.b_{1,1,2}\emph{\emph{D}}_{1,2}(b_{2,1,2})
+b_{1,1,2}\emph{\emph{D}}_{1,3}(b_{2,0,1})\,u\right]\,u_{x^2}+b_{1,1,1}\emph{\emph{D}}_{1,3}(b_{2,1,1})u_{x^1}^2\nonumber\\
&+& \left[ b_{1,1,1}\emph{\emph{D}}_{1,3}(b_{2,1,2})+b_{1,1,2}\emph{\emph{D}}_{1,3}(b_{2,1,1})   \right]\,u_{x^1}u_{x^2}  + b_{1,1,2}\emph{\emph{D}}_{1,3}(b_{2,1,2})u_{x^2}^2\nonumber\\
&+&b_{1,1,1}b_{2,1,1}\,u_{2x^1}+\left[ b_{1,1,2}b_{2,1,1}+b_{1,1,1}b_{2,1,2}  \right]\,u_{x^1x^2}+     b_{1,1,2}b_{2,1,2}  \,u_{2x^2}.
\end{eqnarray}
Identifying (\ref{nleqaa11}) with (\ref{nleqaa12}) yields
\begin{proposition}
A necessary and  sufficient condition for the differential operator $\mathcal{P}(2) $ defined by (\ref{nleqaa9})  be decomposed into
the form (\ref{nleqaa10}) is:
\begin{eqnarray}
g_{2,1} &=& b_{1,1,1}b_{2,1,1},
\label{nleqaa13}\\
g_{2,2}+ g_{2,3}&=& b_{1,1,2}b_{2,1,1}+b_{1,1,1}b_{2,1,2},\label{nleqaa14}\\
g_{2,4} &=& b_{1,1,2}b_{2,1,2},\label{nleqaa15}\\
g_{1,1} &=&b_{1,0,1}b_{2,1,1}+b_{1,1,1}b_{2,0,1}+\mathcal{L}(b_{2,1,1})+b_{1,1,1}\emph{\emph{D}}_{1,3}(b_{2,0,1})\,u,\label{nleqaa16}\\
g_{1,2} &=&b_{1,0,1}b_{2,1,2}+b_{1,1,2}b_{2,0,1}+\mathcal{L}(b_{2,1,2})+b_{1,1,2}\emph{\emph{D}}_{1,3}(b_{2,0,1})\,u,\label{nleqaa17}\\
g_{0,1} &=&b_{1,0,1}b_{2,0,1}+\mathcal{L}(b_{2,0,2}),\label{nleqaa18}\\
0&=&b_{1,1,2}\emph{\emph{D}}_{1,3}(b_{2,1,2}),\\
0&=&b_{1,1,1}\emph{\emph{D}}_{1,3}(b_{2,1,1}),\\
0&=& b_{1,1,1}\emph{\emph{D}}_{1,3}(b_{2,1,2})+b_{1,1,2}\emph{\emph{D}}_{1,3}(b_{2,1,1}),
\end{eqnarray}
where $\mathcal{L}=b_{1,1,1}\emph{\emph{D}}_{1,1}+b_{1,1,2}\emph{\emph{D}}_{1,2}.$
\end{proposition}

\subsection{Factorizations of systems of nonlinear  differential equations}
\subsubsection{Theoretical considerations and principles}

Let $\Lambda$ be an open subset of $\mathbb{R}^n$ and $\Omega,$ an open subset of $\mathbb{R}^m.$
Examine now the factorization process for systems of $s$-th order, $(s\geq 2),$ nonlinear
 differential equations with $n$ independent variables $x=\left(x^1, \cdots ,x^n\right)$
 and $m\geq 2$ dependent variables $u=\,^{t}\left(u^1, \cdots ,u^m\right),$ $u=u(x)$
 whose associated matrix operator, $\mathcal{M}(s),$ is of the form
 \begin{equation}\label{nleqb1}
\mathcal{M}(s)=\left[  \mathcal{R}_{p,q}\left(s_{p,q}\right)   \right]_{1\leq p,q\leq m};
 \end{equation}
the   $\mathcal{R}_{p,q}\left(s_{p,q}\right)$ are $s_{p,q}$-th order linear differential operators
 \begin{equation}\label{nleqb2}
\mathcal{R}_{p,q}\left(s_{p,q}\right)=\sum_{k=0}^{s_{p,q}}\sum_{h=1}^{p_k}f_{p,q,k,h}(x,\underbrace{\cdot,\ldots,\cdot}_{m\emph{\emph{-entries}}})\emph{\emph{D}}_{k,h},
 \end{equation}
where $f_{p,q,k,h}\in \mathcal{C}(\Lambda\times\Omega,\mathbb{R}),$ $s_{p,q}=s-1+\delta_{p,q},$ $\delta_{p,p}=1$ and $\delta_{p,q}= 0$ if $p\neq q.$  \\
Let $\Lambda$ and $\Lambda_0$ be two open subsets of $\mathbb{R}^n$ such that $\Lambda_0 \subset \Lambda.$ The matrix operator $\mathcal{M}(s)$ acts on a vector valued function $u=\,^{t}\left(u^1, \cdots ,u^m\right) \in \mathcal{C}^s(\Lambda_0,\Omega)$ as follows
$$
\mathcal{M}(s) \,u=\left[  \mathcal{R}_{p,q}\left(s_{p,q}\right)   \right]_{1\leq p,q\leq m}\,u=\left[  \sum_{q=1}^{m} \mathcal{R}_{p,q}\left(s_{p,q}\right)  \,u^q \right]_{1\leq p\leq m},
$$
with
\begin{equation}
\mathcal{R}_{p,q}\left(s_{p,q}\right)\,u^q=\sum_{k=0}^{s_{p,q}}\sum_{h=1}^{p_k}f_{p,q,k,h}\left(x,u^1,\ldots,u^m\right)\emph{\emph{D}}_{k,h}\,u^q.
 \end{equation}
The method of factorization consists in seeking a decomposition of the matrix   $\mathcal{M}(s)$ under the following form
 \begin{equation}\label{nleqb3}
\mathcal{M}(s)=\prod_{i=1}^{l}\mathcal{N}_i(s_i)
 \end{equation}
where
\begin{equation}\label{nleqb4}
\mathcal{N}_i(s_i)=\left[  \mathcal{T}_{i,p,q}\left(s_{i,p,q}\right)   \right]_{1\leq p,q\leq m}
 \end{equation}
  and
\begin{equation}\label{nleqb5}
\mathcal{T}_{i,p,q}\left(s_{i,p,q}\right)=\sum_{k=0}^{s_{i,p,q}}\sum_{h=1}^{p_k}a_{i,p,q,k,h}(x,\underbrace{\cdot,\ldots,\cdot}_{m\emph{\emph{-entries}}})\emph{\emph{D}}_{k,h},
 \end{equation}
with $\sum_{i=1}^l s_i=s,$   $s_{i,p,q}=s_i-1+\delta_{p,q},\,$
$a_{1,p,q,k,h} \in \mathcal{C}(\Lambda\times\Omega, \mathbb{R})$
 and  $a_{i,p,q,k,h} \in \mathcal{C}^{\sum_{j=1}^{i-1}s_{i,p,q}}(\Lambda\times\Omega, \mathbb{R}),$
$i=2,3, \cdots  ,l.$\\
Expanding (\ref{nleqb3}) leads to the relations between the unknown functions   $a_{i,p,q,k,h}$ of $\mathcal{N}_i(s_i)$
and the known functions $f_{p,q,k,h}$ of $\mathcal{M}(s).$

\subsubsection{Necessary and sufficient conditions for the factorization of systems of second order nonlinear ODEs}

Let  $\Lambda,$  $\Lambda_0$ be two open subsets of $\mathbb{R}$ such that $\Lambda_0 \subset \Lambda,$
and $\Omega$ an open subset of $\mathbb{R}^m.$
Consider the matrix operator
\begin{equation}\label{nleqbb1}
\mathcal{M}(2)=\left[  \mathcal{R}_{p,q}  \right]_{1\leq p,q\leq m},
 \end{equation}
where
\begin{eqnarray}\label{nleqbb2}
\mathcal{R}_{p,p}&=&\sum_{k=0}^{2}\sum_{h=1}^{p_k}f_{p,q,k,h}(x,\underbrace{\cdot,\ldots,\cdot}_{m\emph{\emph{-entries}}})\emph{\emph{D}}_{k,h}\\
&=&f_{p,p,0,1}(x,\underbrace{\cdot,\ldots,\cdot}_{m\emph{\emph{-entries}}})\emph{\emph{D}}_{0,1}+f_{p,p,1,1}(x,\underbrace{\cdot,\ldots,\cdot}_{m\emph{\emph{-entries}}})\emph{\emph{D}}_{1,1}+f_{p,p,2,1}(x,\underbrace{\cdot,\ldots,\cdot}_{m\emph{\emph{-entries}}})\emph{\emph{D}}_{2,1}\nonumber
 \end{eqnarray}
and for $p \neq q$
\begin{eqnarray}\label{nleqbb3}
\mathcal{R}_{p,q}&=&\sum_{k=0}^{1}\sum_{h=1}^{p_k}f_{p,q,k,h}(x,\underbrace{\cdot,\ldots,\cdot}_{m\emph{\emph{-entries}}})\emph{\emph{D}}_{k,h} \nonumber\\ &=&f_{p,q,0,1}(x,\underbrace{\cdot,\ldots,\cdot}_{m\emph{\emph{-entries}}})\emph{\emph{D}}_{0,1}+f_{p,q,1,1}(x,\underbrace{\cdot,\ldots,\cdot}_{m\emph{\emph{-entries}}})\emph{\emph{D}}_{1,1}
 \end{eqnarray}
with $f_{p,q,k,h}\in \mathcal{C}(\Lambda\times\Omega,\mathbb{R}),$ $x=x^1.$ Write $\mathcal{M}(2)$ in the form
\begin{equation}\label{nleqbb4}
\mathcal{M}(2)=\mathcal{N}_1(1)\cdot \mathcal{N}_2(1),
 \end{equation}
where
\begin{equation}\label{nleqbb5}
\mathcal{N}_i(1)=\left[  \mathcal{T}_{i,p,q}  \right]_{1\leq p,q\leq m}
 \end{equation}
with
\begin{eqnarray}\label{nleqbb6}
\mathcal{T}_{i,p,p}&=&\sum_{k=0}^{1}\sum_{h=1}^{p_k}a_{i,p,p,k,h}(x,\underbrace{\cdot,\ldots,\cdot}_{m\emph{\emph{-entries}}})\emph{\emph{D}}_{k,h}\nonumber\\ &=&a_{i,p,p,0,1}(x,\underbrace{\cdot,\ldots,\cdot}_{m\emph{\emph{-entries}}})\emph{\emph{D}}_{0,1}+a_{i,p,p,1,1}(x,\underbrace{\cdot,\ldots,\cdot}_{m\emph{\emph{-entries}}})\emph{\emph{D}}_{1,1}
 \end{eqnarray}
and for $p \neq q$
\begin{equation}\label{nleqbb7}
\mathcal{T}_{i,p,q} =a_{i,p,q,0,1}(x,\underbrace{\cdot,\ldots,\cdot}_{m\emph{\emph{-entries}}})\emph{\emph{D}}_{0,1},
 \end{equation}
$a_{1,p,q,k,h}\in \mathcal{C}(\Lambda\times\Omega,\mathbb{R})$ and $a_{2,p,q,k,h}\in \mathcal{C}^1(\Lambda\times\Omega,\mathbb{R}).$
Let $u =\,^{t}\left(u^1, \cdots ,u^m\right) \in \mathcal{C}^2(\Lambda_0,\Omega).$ Then we have
\begin{equation}\label{nleqbb8}
\mathcal{M}(2)\,u = \left[  \mathcal{R}_{p,q}  \right]_{1\leq p,q\leq m}\,u=\left[ \sum_{q=1}^m \mathcal{R}_{p,q}\,u^q  \right]_{1\leq p\leq m}
\end{equation}
where
$$  \mathcal{R}_{p,p}\,u^p= f_{p,p,0,1}(x,u)\,u^p+f_{p,p,1,1}(x,u)\,u^p_x+f_{p,p,2,1}(x,u)\,u^p_{2x} $$
and for $p \neq q$
$$ \mathcal{R}_{p,q}\,u^q =f_{p,q,0,1}(x,u)\,u^q+f_{p,q,1,1}(x,u)\,u^q_x.  $$
On the other hand, after expansion of (\ref{nleqbb4}), we have
\begin{equation}\label{nleqbb9}
\mathcal{M}(2)\,u = \left[  \widetilde{\mathcal{R}}_{p,q}  \right]_{1\leq p,q\leq m}\,u=\left[ \sum_{q=1}^m \widetilde{\mathcal{R}}_{p,q}\,u^q  \right]_{1\leq p\leq m},
\end{equation}
where
\begin{eqnarray}
\widetilde{\mathcal{R}}_{p,p}\,u^p&=&a_{1,p,p,1,1}a_{2,p,p,1,1}\,u^p_{2x} +a_{1,p,p,1,1}\sum_{\widetilde{h}=1}^m\emph{\emph{D}}_{1,\widetilde{h}+1}(a_{2,p,p,1,1})\,u^{\widetilde{h}}_{x}\,u^p_x\nonumber\\
&+&a_{1,p,p,1,1}\sum_{^{\widetilde{h}=1}_{\widetilde{h}\neq p}}^m\emph{\emph{D}}_{1,\widetilde{h}+1}(a_{2,p,p,0,1})\,u^{\widetilde{h}}_{x}\,u^p
+ \left[a_{1,p,p,0,1}a_{2,p,p,1,1}+a_{1,p,p,1,1}a_{2,p,p,0,1}\right.\nonumber\\
&+&\left.a_{1,p,p,1,1}\emph{\emph{D}}_{1,1}(a_{2,p,p,1,1})+a_{1,p,p,1,1}\emph{\emph{D}}_{1,p+1}(a_{2,p,p,0,1})\,u^p \right]u^p_x\nonumber\\
&+&\left[\sum_{l=1}^m a_{1,p,l,0,1}a_{2,l,p,0,1}+a_{1,p,p,1,1}\emph{\emph{D}}_{1,1}(a_{2,p,p,0,1}) \right]u^p\nonumber
\end{eqnarray}
and for $p \neq q$
\begin{eqnarray}
\widetilde{\mathcal{R}}_{p,q}\,u^q&=&\left[a_{1,p,p,1,1}a_{2,p,q,0,1}+a_{1,p,q,0,1}a_{2,q,q,1,1}+a_{1,p,p,1,1}\emph{\emph{D}}_{1,q+1}(a_{2,p,q,0,1})\,u^q \right]u^q_x\nonumber\\
&+&a_{1,p,p,1,1}\sum_{^{\widetilde{h}=1}_{\widetilde{h}\neq q}}^m\emph{\emph{D}}_{1,\widetilde{h}+1}(a_{2,p,q,0,1})\,u^{\widetilde{h}}_{x}\,u^q\nonumber\\
&+&\left[\sum_{l=1}^m a_{1,p,l,0,1}a_{2,l,q,0,1}+a_{1,p,p,1,1}\emph{\emph{D}}_{1,1}(a_{2,p,q,0,1}) \right]u^q.\nonumber
\end{eqnarray}
Identifying (\ref{nleqbb8}) with (\ref{nleqbb9}) yields
\begin{proposition}
A necessary and  sufficient condition for the differential operator $\mathcal{M}(2) $ defined by (\ref{nleqbb1})  be decomposed into
the form (\ref{nleqbb4}) is:
\begin{eqnarray}
f_{p,p,0,1}&=&\sum_{l=1}^m a_{1,p,l,0,1}a_{2,l,p,0,1}+a_{1,p,p,1,1}\emph{\emph{D}}_{1,1}(a_{2,p,p,0,1}),\\
f_{p,p,1,1}&=& a_{1,p,p,0,1}a_{2,p,p,1,1}+a_{1,p,p,1,1}a_{2,p,p,0,1}\nonumber\\
&+&a_{1,p,p,1,1}\emph{\emph{D}}_{1,1}(a_{2,p,p,1,1})+a_{1,p,p,1,1}\emph{\emph{D}}_{1,p+1}(a_{2,p,p,0,1})\,u^p,\\
f_{p,p,2,1}    &=&a_{1,p,p,1,1}a_{2,p,p,1,1},\\
0&=&\emph{\emph{D}}_{1,\widetilde{h}+1}(a_{2,p,p,0,1}),  \quad \widetilde{h}\in\{ 1,2,\cdots,m  \}\setminus\{p\}, \\
0&=&\emph{\emph{D}}_{1,\widetilde{h}+1}(a_{2,p,p,1,1}),\quad \widetilde{h}=1,2,\cdots,m
\end{eqnarray}
and for $p \neq q$
\begin{eqnarray}
f_{p,q,0,1}&=&\sum_{l=1}^m a_{1,p,l,0,1}a_{2,l,q,0,1}+a_{1,p,p,1,1}\emph{\emph{D}}_{1,1}(a_{2,p,q,0,1}),\\
f_{p,q,1,1}   &=& a_{1,p,p,1,1}a_{2,p,q,0,1}+a_{1,p,q,0,1}a_{2,q,q,1,1}+a_{1,p,p,1,1}\emph{\emph{D}}_{1,q+1}(a_{2,p,q,0,1})\,u^q,\\
0&=&\emph{\emph{D}}_{1,\widetilde{h}+1}(a_{2,p,q,0,1}),  \quad \widetilde{h}\in\{ 1,2,\cdots,m  \}\setminus\{q\}.
\end{eqnarray}
\end{proposition}

\subsubsection{Necessary and sufficient conditions for the factorization of systems of second order nonlinear PDEs with two independent variables}

Let $\Lambda,$  $\Lambda_0$ be two open subsets of $\mathbb{R}^2$ such that $\Lambda_0 \subset \Lambda,$ and
$\Omega$ an open subset of $\mathbb{R}^m.$
Consider the matrix operator
\begin{equation}\label{nleqbbb1}
\mathcal{M}(2)=\left[  \mathcal{R}_{p,q}  \right]_{1\leq p,q\leq m},
 \end{equation}
where
\begin{eqnarray}\label{nleqbbb2}
\mathcal{R}_{p,p}&=&\sum_{k=0}^{2}\sum_{h=1}^{p_k}f_{p,q,k,h}(x,\underbrace{\cdot,\ldots,\cdot}_{m\emph{\emph{-entries}}})\emph{\emph{D}}_{k,h}\\ &=&f_{p,p,0,1}(x,\underbrace{\cdot,\ldots,\cdot}_{m\emph{\emph{-entries}}})\emph{\emph{D}}_{0,1}+f_{p,p,1,1}(x,\underbrace{\cdot,\ldots,\cdot}_{m\emph{\emph{-entries}}})\emph{\emph{D}}_{1,1}
+f_{p,p,1,2}(x,\underbrace{\cdot,\ldots,\cdot}_{m\emph{\emph{-entries}}})\emph{\emph{D}}_{1,2}\nonumber\\
&+&f_{p,p,2,1}(x,\underbrace{\cdot,\ldots,\cdot}_{m\emph{\emph{-entries}}})\emph{\emph{D}}_{2,1}
+f_{p,p,2,2}(x,\underbrace{\cdot,\ldots,\cdot}_{m\emph{\emph{-entries}}})\emph{\emph{D}}_{2,2}\nonumber\\
&+&f_{p,p,2,3}(x,\underbrace{\cdot,\ldots,\cdot}_{m\emph{\emph{-entries}}})\emph{\emph{D}}_{2,3}
+f_{p,p,2,4}(x,\underbrace{\cdot,\ldots,\cdot}_{m\emph{\emph{-entries}}})\emph{\emph{D}}_{2,4}\nonumber
 \end{eqnarray}
and for $p \neq q$
\begin{eqnarray}\label{nleqbbb3}
\mathcal{R}_{p,q}&=&\sum_{k=0}^{1}\sum_{h=1}^{p_k}f_{p,q,k,h}(x,\underbrace{\cdot,\ldots,\cdot}_{m\emph{\emph{-entries}}})\emph{\emph{D}}_{k,h}\\
&=&f_{p,q,0,1}(x,\underbrace{\cdot,\ldots,\cdot}_{m\emph{\emph{-entries}}})\emph{\emph{D}}_{0,1}+f_{p,q,1,1}(x,\underbrace{\cdot,\ldots,\cdot}_{m\emph{\emph{-entries}}})\emph{\emph{D}}_{1,1}
+f_{p,q,1,2}(x,\underbrace{\cdot,\ldots,\cdot}_{m\emph{\emph{-entries}}})\emph{\emph{D}}_{1,2}\nonumber
 \end{eqnarray}
with $f_{p,q,k,h}\in \mathcal{C}(\Lambda\times\Omega,\mathbb{R}),$ $x= \left(x^1,x^2\right).$
Write $\mathcal{M}(2)$ in the form
\begin{equation}\label{nleqbbb4}
\mathcal{M}(2)=\mathcal{N}_1(1)\cdot \mathcal{N}_2(1),
 \end{equation}
where
\begin{equation}\label{nleqbbb5}
\mathcal{N}_i(1)=\left[  \mathcal{T}_{i,p,q}  \right]_{1\leq p,q\leq m}
 \end{equation}
with
\begin{eqnarray}\label{nleqbbb6}
\mathcal{T}_{i,p,p}&=&\sum_{k=0}^{1}\sum_{h=1}^{p_k}a_{i,p,p,k,h}(x\underbrace{\cdot,\ldots,\cdot}_{m\emph{\emph{-entries}}})\emph{\emph{D}}_{k,h}\\ &=&a_{i,p,p,0,1}(x\underbrace{\cdot,\ldots,\cdot}_{m\emph{\emph{-entries}}})\emph{\emph{D}}_{0,1}+a_{i,p,p,1,1}(x\underbrace{\cdot,\ldots,\cdot}_{m\emph{\emph{-entries}}})\emph{\emph{D}}_{1,1}
+a_{i,p,p,1,2}(x\underbrace{\cdot,\ldots,\cdot}_{m\emph{\emph{-entries}}})\emph{\emph{D}}_{1,2}\nonumber
 \end{eqnarray}
and for $p \neq q$
\begin{equation}\label{nleqbbb7}
\mathcal{T}_{i,p,q} =a_{i,p,q,0,1}(x\underbrace{\cdot,\ldots,\cdot}_{m\emph{\emph{-entries}}})\emph{\emph{D}}_{0,1},
 \end{equation}
$a_{1,p,q,k,h}\in \mathcal{C}(\Lambda\times\Omega,\mathbb{R})$ and $a_{2,p,q,k,h}\in \mathcal{C}^1(\Lambda\times\Omega,\mathbb{R}).$
Let $u =\,^{t}\left(u^1, \cdots ,u^m\right) \in \mathcal{C}^2(\Lambda_0,\Omega).$ Then we have
\begin{equation}\label{nleqbbb8}
\mathcal{M}(2)\,u = \left[  \mathcal{R}_{p,q}  \right]_{1\leq p,q\leq m}\,u=\left[ \sum_{q=1}^m \mathcal{R}_{p,q}\,u^q  \right]_{1\leq p\leq m},
\end{equation}
where
\begin{eqnarray}
  \mathcal{R}_{p,p}\,u^p&=& f_{p,p,0,1}\,u^p+f_{p,p,1,1}\,u^p_{x^1}+f_{p,p,1,2}\,u^p_{x^2}\nonumber\\
  &+&f_{p,p,2,1}\,u^p_{2x^1}
+\left(f_{p,p,2,2}+f_{p,p,2,3}\right)\,u^p_{x^1x^2}+f_{p,p,2,4}\,u^p_{2x^2}
\end{eqnarray}
and for $p \neq q$
$$ \mathcal{R}_{p,q}\,u^q =f_{p,q,0,1}\,u^q+f_{p,q,1,1}\,u^q_{x^1}+f_{p,q,1,2}\,u^q_{x^2}.  $$
On the other hand, after expansion of (\ref{nleqbbb4}), we have
\begin{equation}\label{nleqbbb9}
\mathcal{M}(2)\,u = \left[  \widetilde{\mathcal{R}}_{p,q}  \right]_{1\leq p,q\leq m}\,u=\left[ \sum_{q=1}^m \widetilde{\mathcal{R}}_{p,q}\,u^q  \right]_{1\leq p\leq m},
\end{equation}
where
\begin{eqnarray}
\widetilde{\mathcal{R}}_{p,p}\,u^p&=&a_{1,p,p,1,1}a_{2,p,p,1,1}\,u^p_{2x^1}+\left(a_{1,p,p,1,2}a_{2,p,p,1,1}+a_{1,p,p,1,1}a_{2,p,p,1,2}\right)\,u^p_{x^1x^2}\nonumber\\
&+&a_{1,p,p,1,2}a_{2,p,p,1,2}\,u^p_{2x^2} \left[a_{1,p,p,0,1}a_{2,p,p,1,1}+a_{1,p,p,1,1}a_{2,p,p,0,1}\right.\nonumber\\
&+&\left.\mathcal{L}_p(a_{2,p,p,1,1})+a_{1,p,p,1,1}\emph{\emph{D}}_{1,p+2}(a_{2,p,p,0,1})\,u^p \right]u^p_{x^1}
+\left[a_{1,p,p,0,1}a_{2,p,p,1,2}\right.\nonumber\\
&+&\left.a_{1,p,p,1,2}a_{2,p,p,0,1}+\mathcal{L}_p(a_{2,p,p,1,2})+a_{1,p,p,1,2}\emph{\emph{D}}_{1,p+2}(a_{2,p,p,0,1})\,u^p \right]u^p_{x^2}\nonumber\\
&+&   a_{1,p,p,1,1}\sum_{\widetilde{h}=1}^m\emph{\emph{D}}_{1,\widetilde{h}+2}(a_{2,p,p,1,1})\,u^{\widetilde{h}}_{x^1}\,u^p_{x^1}           +a_{1,p,p,1,1}\sum_{\widetilde{h}=1}^m\emph{\emph{D}}_{1,\widetilde{h}+2}(a_{2,p,p,1,2})\,u^{\widetilde{h}}_{x^1}\,u^p_{x^2}\nonumber\\
&+&   a_{1,p,p,1,2}\sum_{\widetilde{h}=1}^m\emph{\emph{D}}_{1,\widetilde{h}+2}(a_{2,p,p,1,1})\,u^{\widetilde{h}}_{x^2}\,u^p_{x^1}           +a_{1,p,p,1,2}\sum_{\widetilde{h}=1}^m\emph{\emph{D}}_{1,\widetilde{h}+2}(a_{2,p,p,1,2})\,u^{\widetilde{h}}_{x^2}\,u^p_{x^2}\nonumber\\
&+&  a_{1,p,p,1,1}\sum_{^{\widetilde{h}=1}_{\widetilde{h}\neq p}}^m\emph{\emph{D}}_{1,\widetilde{h}+2}(a_{2,p,p,0,1})\,u^{\widetilde{h}}_{x^1}\,u^p  +a_{1,p,p,1,2}\sum_{^{\widetilde{h}=1}_{\widetilde{h}\neq p}}^m\emph{\emph{D}}_{1,\widetilde{h}+2}(a_{2,p,p,0,1})\,u^{\widetilde{h}}_{x^2}\,u^p          \nonumber\\
&+&\left[\sum_{l=1}^m a_{1,p,l,0,1}a_{2,l,p,0,1}+\mathcal{L}_p(a_{2,p,p,0,1}) \right]u^p\nonumber
\end{eqnarray}
and for $p \neq q$
\begin{eqnarray}
\widetilde{\mathcal{R}}_{p,q}\,u^q&=&\left[a_{1,p,p,1,1}a_{2,p,q,0,1}+a_{1,p,q,0,1}a_{2,q,q,1,1}+a_{1,p,p,1,1}\emph{\emph{D}}_{1,q+2}(a_{2,p,q,0,1})\,u^q \right]u^q_{x^1}  \nonumber\\
&+&\left[a_{1,p,p,1,2}a_{2,p,q,0,1}+a_{1,p,q,0,1}a_{2,q,q,1,2}+a_{1,p,p,1,2}\emph{\emph{D}}_{1,q+2}(a_{2,p,q,0,1})\,u^q \right]u^q_{x^2}   \nonumber\\
 &+&a_{1,p,p,1,1}\sum_{^{\widetilde{h}=1}_{\widetilde{h}\neq q}}^m\emph{\emph{D}}_{1,\widetilde{h}+2}(a_{2,p,q,0,1})\,u^{\widetilde{h}}_{x^1}\,u^q    +a_{1,p,p,1,2}\sum_{^{\widetilde{h}=1}_{\widetilde{h}\neq q}}^m\emph{\emph{D}}_{1,\widetilde{h}+2}(a_{2,p,q,0,1})\,u^{\widetilde{h}}_{x^2}\,u^q \nonumber\\
&+&\left[\sum_{l=1}^m a_{1,p,l,0,1}a_{2,l,q,0,1}+\mathcal{L}_p(a_{2,p,q,0,1}) \right]u^q,\nonumber
\end{eqnarray}
where $\mathcal{L}_p=a_{1,p,p,1,1}\emph{\emph{D}}_{1,1}+a_{1,p,p,1,2}\emph{\emph{D}}_{1,2}.$
Identifying (\ref{nleqbbb8}) with (\ref{nleqbbb9}) yields
\begin{proposition}
A necessary and  sufficient condition for the differential operator $\mathcal{M}(2) $ defined by (\ref{nleqbbb1})  be decomposed into
the form (\ref{nleqbbb4}) is:
\begin{eqnarray}
f_{p,p,0,1}&=&\sum_{l=1}^m a_{1,p,l,0,1}a_{2,l,p,0,1}+\mathcal{L}_p(a_{2,p,p,0,1}),\\
f_{p,p,1,1}&=& a_{1,p,p,0,1}a_{2,p,p,1,1}+a_{1,p,p,1,1}a_{2,p,p,0,1}\nonumber\\
&+&\mathcal{L}_p(a_{2,p,p,1,1})+a_{1,p,p,1,1}\emph{\emph{D}}_{1,p+2}(a_{2,p,p,0,1})\,u^p,\\
f_{p,p,1,2}&=& a_{1,p,p,0,1}a_{2,p,p,1,2}+a_{1,p,p,1,2}a_{2,p,p,0,1}\nonumber\\
&+&\mathcal{L}_p(a_{2,p,p,1,2})+a_{1,p,p,1,2}\emph{\emph{D}}_{1,p+2}(a_{2,p,p,0,1})\,u^p,\\
f_{p,p,2,1}    &=&a_{1,p,p,1,1}a_{2,p,p,1,1},\\
f_{p,p,2,2}+f_{p,p,2,3}    &=&a_{1,p,p,1,2}a_{2,p,p,1,1}+a_{1,p,p,1,1}a_{2,p,p,1,2},\\
f_{p,p,2,4}    &=&a_{1,p,p,1,2}a_{2,p,p,1,2},\\
0&=&\emph{\emph{D}}_{1,\widetilde{h}+2}(a_{2,p,p,0,1}),  \quad \widetilde{h}\in\{ 1,2,\cdots,m  \}\setminus\{p\}, \\
0&=&\emph{\emph{D}}_{1,\widetilde{h}+2}(a_{2,p,p,1,1}),\quad \widetilde{h}=1,2,\cdots,m,\\
0&=&\emph{\emph{D}}_{1,\widetilde{h}+2}(a_{2,p,p,1,2}),\quad \widetilde{h}=1,2,\cdots,m
\end{eqnarray}
and for $p \neq q$
\begin{eqnarray}
f_{p,q,0,1}&=&\sum_{l=1}^m a_{1,p,l,0,1}a_{2,l,q,0,1}+\mathcal{L}_p(a_{2,p,q,0,1}),\\
f_{p,q,1,1}   &=& a_{1,p,p,1,1}a_{2,p,q,0,1}+a_{1,p,q,0,1}a_{2,q,q,1,1}+a_{1,p,p,1,1}\emph{\emph{D}}_{1,q+2}(a_{2,p,q,0,1})\,u^q,\\
f_{p,q,1,2}   &=& a_{1,p,p,1,2}a_{2,p,q,0,1}+a_{1,p,q,0,1}a_{2,q,q,1,2}+a_{1,p,p,1,2}\emph{\emph{D}}_{1,q+2}(a_{2,p,q,0,1})\,u^q,\\
0&=&\emph{\emph{D}}_{1,\widetilde{h}+2}(a_{2,p,q,0,1}),  \quad \widetilde{h}\in\{ 1,2,\cdots,m  \}\setminus\{q\}.
\end{eqnarray}
\end{proposition}

\subsection*{Acknowledgments}
This work is partially supported by the ICTP through the
OEA-ICMPA-Prj-15. The ICMPA is in partnership with
the Daniel Iagolnitzer Foundation (DIF), France.

\label{lastpage-01}
\end{document}